\documentclass{amsart}

\usepackage{amssymb}
\usepackage[all,cmtip]{xy}
\usepackage{tikz-cd}
\usetikzlibrary{matrix}
\usepackage{hyperref}

\newcommand{\isomto}{\overset{\sim}{\to}}

\newcommand{\injto}{\hookrightarrow}

\makeatletter
\newcommand*\rel@kern[1]{\kern#1\dimexpr\macc@kerna}
\newcommand*\widebar[1]{%
  \begingroup
  \def\mathaccent##1##2{%
    \rel@kern{0.8}%
    \overline{\rel@kern{-0.8}\macc@nucleus\rel@kern{0.2}}%
    \rel@kern{-0.2}%
  }%
  \macc@depth\@ne
  \let\math@bgroup\@empty \let\math@egroup\macc@set@skewchar
  \mathsurround\z@ \frozen@everymath{\mathgroup\macc@group\relax}%
  \macc@set@skewchar\relax
  \let\mathaccentV\macc@nested@a
  \macc@nested@a\relax111{#1}%
  \endgroup
}
\makeatother

\DeclareMathOperator\Hom{Hom}
\DeclareMathOperator\Frob{Frob}
\DeclareMathOperator\End{End}
\DeclareMathOperator\Spec{Spec}

\DeclareMathOperator\Pic{Pic}

\DeclareMathOperator\Gal{Gal}

\DeclareMathOperator\Br{Br}
\DeclareMathOperator\Brhat{\hat Br}
\DeclareMathOperator\Nm{Nm}

\DeclareMathOperator\tr{tr}

\DeclareMathOperator\rk{rk}
\DeclareMathOperator\disc{disc}

\def\bQ{{\mathbf{Q}}} \def\bZ{{\mathbf{Z}}} 
\def\bF{{\mathbf{F}}}  \def\bR{{\mathbf{R}}}
\def\bC{{\mathbf{C}}}  
\def\bA{{\mathbf{A}}}  \def\bD{{\mathbf{D}}}

\def\cO{{\mathcal{O}}}

\def\cK{{\mathcal K}} 

  \def\rH{{\rm H}} 
  \def\rW{{\rm W}}

 \def\fX{{\mathfrak{X}}}

\def\et{{\rm et}}

\def\crys{{\rm crys}}
\def\HS{{\rm HS}}
\def\NS{{\rm NS}}
\def\alg{{\rm alg}}
\def\trc{{\rm trc}}
\def\Sh{{\rm Sh}}

\theoremstyle{plain}
\newtheorem{theorem}{Theorem}

\newtheorem{lemma}{Lemma}

\newtheorem{proposition}{Proposition}
\theoremstyle{definition}
\newtheorem{definition}{Definition}

\newtheorem{remark}{Remark}

\begin{document}

\title[K3 surfaces over finite fields with given $L$-function]
 {K3 surfaces over finite fields \\ with given $L$-function}

\author
 {Lenny Taelman}

\maketitle

\begin{abstract}
The zeta function of a K3 surface over a finite field satisfies a number of obvious (archimedean and $\ell$-adic) and a number of less obvious ($p$-adic) constraints. We consider the converse question, in the style of Honda-Tate: given a function $Z$ satisfying all these constraints, does there exist a K3 surface whose zeta-function equals $Z$? Assuming semi-stable reduction, we show that the answer is yes if we allow a finite extension of the finite field. An important ingredient in the proof is the construction of complex projective K3 surfaces with complex multiplication by a given CM field. 
\end{abstract}

\section*{Introduction}

Let $X$ be a K3 surface over $\bF_q$. The zeta function of $X$ has the form
\[
	Z(X/\bF_q,T) = \frac{1}{(1-T) L(X/\bF_q,qT) (1-q^2T) }
\]
where the polynomial $L(X/\bF_q)$ is defined by
\[
	L(X/\bF_q,T) := \det( 1-T\Frob, \rH^2(X_{\bar\bF_q}, \bQ_\ell(1)) )  \in \bQ[T].
\]
We have $L(X/\bF_q,T) = \prod_{i=1}^{22} (1-\gamma_i T)$ with the $\gamma_i$ of complex absolute value $1$.  The polynomial $L(X/\bF_q,T)$ factors in $\bQ[T]$ as  $L=L_\alg L_\trc$ with
\[
	L_\alg(X/\bF_q, T) = \!\prod_{\gamma_i \in \mu_\infty} \!(1-T\gamma_i), \quad\quad
	L_\trc(X/\bF_q, T) = \!\prod_{\gamma_i \not\in \mu_\infty} \!(1-T\gamma_i),
\]
where $\mu_\infty$ is the group of complex roots of unity. 

\begin{theorem}\label{thm:properties}
 Let $X$ be a K3 surface over $\bF_q$ with  $q=p^a$. Assume that $X$ is not supersingular. Then 
\begin{enumerate}
\item all complex roots of $L_\trc(X/\bF_q,T)$ have absolute value $1$;
\item no root of $L_\trc(X/\bF_q,T)$ is a root of unity;
\item $L_\trc(X/\bF_q,T) \in \bZ_\ell[T]$ for all $\ell \neq p$; 
\item the Newton polygon of $L_\trc(X/\bF_q,T)$ at $p$ is of the form
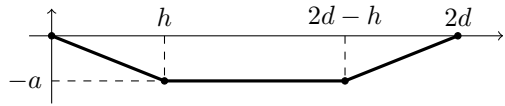
\[
\begin{tikzpicture}[scale=.3]
\draw[->] (-1,0)--(20,0);
\draw[->] (0,-3)--(0,1.2);
\draw[very thick] (0,0)--(5,-2)--(13,-2)--(18,0);
\draw[dashed] (5,0)--(5,-2);
\draw[dashed] (0,-2)--(5,-2);
\draw[dashed] (13,0)--(13,-2);
\draw[fill] (0,0) circle [radius=4pt];
\draw[fill] (5,-2) circle [radius=4pt];
\draw[fill] (13,-2) circle [radius=4pt];
\draw[fill] (18,0) circle [radius=4pt];
\draw (5,0) node[above] {$h$};
\draw (13,0) node[above] {$2d-h$};
\draw (18,0) node[above] {$2d$};
\draw (0,-2) node[left] {$-a$};
\end{tikzpicture}
\]
with $h$ and $d$ integers satisfying $1\leq h \leq d \leq 10$;
\item $L_\trc(X/\bF_q,T) = Q^e$ for some $e>0$ and some irreducible $Q\in\bQ[T]$, and
	$Q$ has a unique irreducible factor in $\bQ_p[T]$ with negative slope.
\end{enumerate}
\end{theorem}

The above theorem collects results of Deligne, Artin, Mazur, Yu and Yui, and slightly expands on these, see
\S \ref{sec:p-adic} for the details. The integer $h$ in the theorem is the {height} of $X$ (which is finite by the assumption that $X$ is not supersingular), and assuming the Tate conjecture (which is now known in almost all cases \cite{Charles13,MadapusiPera15,Charles14}) the Picard rank of $X_{\bar\bF_q}$ is $22-2d$.

\begin{definition}[Property ($\star$)] A K3 surface $X$ over a finite extension $k$ of $\bQ_p$ is said to satisfy ($\star$) if there exists a finite extension $k\subset \ell$ and a proper flat algebraic space $\fX \to \Spec \cO_\ell$ such that
\begin{enumerate}
\item $\fX \times_{\Spec \cO_\ell} \Spec \ell \cong X \times_{\Spec k} \Spec \ell$,
\item $\fX$ is regular,
\item the special fiber of $\fX$ is a reduced normal crossings divisor with smooth components,
\item $\omega_{\fX/\cO_\ell} \cong \cO_{\fX}$.
\end{enumerate}
\end{definition}

Property ($\star$) is a strong form of potential semi-stability. It is expected that every $X$ satisfies ($\star$), but this is presently only known for special classes of K3 surfaces, see \cite[\S 4]{Maulik14} and \cite[\S 2]{LiedtkeMatsumoto15}.  Our main result is the following partial converse to Theorem \ref{thm:properties}.

\begin{theorem}\label{thm:converse}
Assume every K3 surface $X$ over a $p$-adic field satisfies $(\star)$. Let 
\[
	L= \prod_{i=1}^{2d} (1-\gamma_iT) \in 1+T \bQ[T]
\]
be a  polynomial which satisfies properties (1)--(5) of Theorem \ref{thm:properties}. Then there exists a positive integer $n$ and a K3 surface $X$ over $\bF_{q^n}$ such that
\[
	L_{\trc}(X/\bF_{q^n},T) = \prod_{i=1}^{2d} (1-\gamma_i^nT).
\]
\end{theorem}

The proof of Theorem \ref{thm:converse} follows the same strategy as the proof of the Honda-Tate theorem \cite{Tate71}: given $L_\trc$, one constructs a K3 surface over a finite field by first producing a complex projective K3 surface with CM by a suitably chosen CM field, then descending it to a number field, and finally reducing it to the residue field at a suitably chosen prime above $p$. In the final step a criterion of good reduction is needed, which has been obtained recently by Matsumoto \cite{Matsumoto15} and Liedtke-Matsumoto \cite{LiedtkeMatsumoto15}, under the assumption ($\star$).

A crucial intermediate result, that may be of independent interest, is the following theorem.

\begin{theorem}\label{thm:K3-with-given-CM-field}
Let $E$ be a CM field with $[E:\bQ]\leq 20$. Then there exists a 
K3 surface over $\bC$ with CM by $E$.
\end{theorem}

See \S \ref{sec:CM-theory} for the definition of `CM by $E$', and see \S \ref{sec:CM-existence} for the proof of this theorem.

\begin{remark} I do not know if one can take $n=1$ in Theorem \ref{thm:converse}.  Finite extensions are used in several parts of the proof, both in constructing a K3 surface $X$ over some finite field, and in verifying that the action of Frobenius on $\rH^2$ is the prescribed one. 

Recently Kedlaya and Sutherland have obtained some computational evidence suggesting that the theorem might hold with $n=1$. They enumerated all polynomials $L$ satisfying (1)--(5) with $q=2$, $\deg L =\deg Q=20$ and with $L(1)=2$ and $L(-1)\neq 2$. There are 1995 such polynomials. If $L=L_\trc(X/\bF_2,T)$ for a K3 surface over $\bF_2$, then the Artin-Tate formula \cite{Milne75,ElsenhansJahnel14} puts strong restrictions on the N\'eron-Severi lattice of $X$. These restrictions suggest that $X$ should be realizable as a smooth quartic, and indeed for each of the 1995 polynomials Kedlaya and Sutherland manage to identify a smooth quartic $X$ defined over $\bF_2$ with $L=L_\trc(X/\bF_2,T)$.

If  one can take $n=1$ in Theorem \ref{thm:converse}, then new ideas will be needed to prove this.
Indeed, there is no reason at all that the $X$ constructed in the current proof is defined over $\bF_q$. 
A similar problem occurs in the proof of the Honda-Tate theorem \cite{Tate71}: given a $q$-Weil number one first constructs an abelian variety over a finite extension of $\bF_q$, and then identifies the desired abelian variety as a simple factor of the Weil restriction to $\bF_q$.  Perhaps a variation of this argument in the context of hyperk\"ahler varieties can be made to work in our setting?
\end{remark}

\begin{remark}By the work of Madapusi Pera \cite{MadapusiPera15}, for every $d$ there is an \'etale map $M_{2d} \to \Sh_{2d}$ from the moduli space of quasi-polarized K3 surfaces of degree $2d$ to a an integral model of a certain Shimura variety, over $\bZ[1/2]$. It is surjective over $\bC$, and assuming ($\star$), one can deduce from the criterion of Liedtke and Matsumoto that it is surjective on $\bar\bF_p$-points. In odd characteristic, Kottwitz \cite{Kottwitz90} and Kisin \cite{Kisin15} have given a group-theoretic description of the isogeny classes in $\Sh_{2d}(\bar\bF_p)$, for every $d$. With arguments similar to those in \S \ref{sec:CM-existence}, it should be possible to  deduce Theorem \ref{thm:properties} and Theorem \ref{thm:converse} from the above results. 
\end{remark}

\subsection*{Acknowledgements} This paper has benefited significantly from enlightening discussions with  Eva Bayer, Christophe Cornut, Johan de Jong, Nick Katz, Mark Kisin, Christian Liedtke, Ronald van Luijk,  Yuya Matsumoto, Davesh Maulik, Ben Moonen and Raman Parimala---I am most grateful for their questions, answers, comments and suggestions. Most of the research leading to this paper was performed while I was a member at the IAS. My gratitude to the Institute and its staff for maintaining the most stimulating research environment cannot be overstated. Special thanks go to Momota Ganguli who promptly and enthusiastically fulfilled even my most obscure bibliographical requests.
 Finally, I wish to thank Lotte Meijer and Michel Reymond for the necessary diversions outside geometry.

The author acknowledges financial support of the National Science Foundation (NSF) and the Netherlands Organization for Scientific Research (NWO).

\section{$p$-adic properties of zeta functions of K3 surfaces}\label{sec:p-adic}

\subsection{Recap on the formal Brauer group of a K3 surface}

Let $X$ be a K3 surface over a field $k$. Artin and Mazur have shown \cite{ArtinMazur77} that the functor
\[
	R \,\mapsto\, \ker\left( \Br X_R \to \Br X \right)
\]
on Artinian $k$-algebras is pro-representable by a (one-dimensional) formal group $\Brhat X$ over $k$. This formal group is called the \emph{formal Brauer group} of $X$.

Assume now that $k$ is a perfect field of characteristic $p>0$ and that $X$ is not supersingular. Then $\Brhat$ has finite height $h$ satisfying $1\leq h \leq 10$. We denote by $\bD(\Brhat X)$ the (covariant) Dieudonn\'e module of $\Brhat X$. This has the structure of an $F$-crystal over $k$. It is free of rank $h$ over the ring $W$ of Witt vectors of $k$.

We denote by $\rH^2_\crys(X/W)_{<1}$ the maximal sub-$F$-crystal of $\rH^2_\crys(X/W)$ that has all slopes $<1$. 

\begin{proposition}\label{prop:crys-Br}
If $X$ is not supersingular, then there is a canonical isomorphism 
\[
	\rH^2_\crys(X/W)_{<1} \cong \bD(\Brhat X)
\]
of $F$-crystals over $k$.
\end{proposition}

\begin{proof}
By \cite[\S 7.2]{Illusie79} there is a canonical isomorphism of $F$-crystals
\begin{equation}\label{eq:dRW-decomposition}
	\rH^2_\crys(X/W) = \rH^2(X,\rW\cO_X) \oplus \rH^1(X, \rW\Omega^1_{X/k})
		\oplus \rH^0(X, \rW \Omega^2_{X/k}),
\end{equation}
coming from the de Rham-Witt complex, and by \cite[Cor 4.3]{ArtinMazur77} we have an isomorphism of $F$-crystals
\[
	\rH^2(X,\rW\cO_X) = \bD(\Brhat X).
\]
Since $\Brhat X$ is a formal group, the slopes of $\rH^2(X,\rW\cO_X) = \bD(\Brhat X)$ are $<1$. On the other hand, since $F$ is divisible by $p^i$ on $\rH^{2-i}(X,\rW\Omega^i_{X/k})$, the slopes of the other summands in (\ref{eq:dRW-decomposition}) are $\geq 1$. This proves the theorem.
\end{proof}

\subsection{Proof of Theorem \ref{thm:properties}}

\begin{proof}
Property (2) holds by definition, (3) is a formal consequence of the trace formula in $\ell$-adic cohomology
(see \emph{e.g.} \cite[\S 1]{Deligne74}), and (1) is part of the Weil conjectures \cite{Deligne72,Deligne74}. 

The other properties make use of crystalline cohomology. Property (4) is well, known. It follows for example from Mazur's proof of `Newton above Hodge' \cite{Mazur72,Mazur73} for liftable varieties with torsion-free cohomology, see \cite[\S2]{Mazur72}. Property (5) is a sharpening of a result of Yu and Yui \cite[Prop.~3.2]{YuYui08}. The argument is essentially the same as in \emph{loc.cit.}, we repeat it for completeness. 

 For a polynomial $Q=\prod (1-\gamma_i T) \in \bQ_p[T]$ we denote by $Q_{<0}$ the product
\[
	Q_{< 0} = \prod_{v_p(\gamma_i)<0} (1-\gamma_i T) \in \bQ_p[T].
\]
Let $K$ be the field of fractions of $W$. If $q=p^a$, then by Proposition \ref{prop:crys-Br} we have
\[
	L_{\trc,<0} := L_{\trc,<0}(X/\bF_q,T) = {\mathrm{det}}_K\big(1-F^aT, \,K\otimes_W\bD(\Brhat X)(1) \big),
\]
in $\bQ_p[T] \subset K[T]$. Since $\Brhat X$ is a one-dimensional formal group of finite height the crystal $\bD(\Brhat X)$ is indecomposable. It follows that the endomorphism $F^a$ of $\bD(\Brhat X)$ has an irreducible minimum polynomial over $K$, and hence $L_{\trc,<0} = P_{<0}^e$ for some irreducible $P_{<0}\in \bQ_p(T)$. Let $Q$ be an irreducible factor of $L_\trc$. Then $Q$ has a reciprocal root $\gamma$ with $v_p(\gamma)<0$, for otherwise the roots of $Q$ would be algebraic integers and hence roots of unity. In particular $Q_{<0}=P_{<0}$. Apparently any two irreducible factors of $L_\trc$ share a common root, hence $L_\trc = Q^e$. This proves (5).
\end{proof}

\section{CM theory of K3 surfaces}\label{sec:CM-theory}

This section collects results of Zarhin, Shafarevich and Rizov.

\subsection{Hodge theoretic aspects}

For a projective K3 surface $X$ over $\bC$ we denote by $\NS(X)$ its N\'eron-Severi group and by $T(X) \subset \rH^2(X,\bZ(1))$ the transcendental lattice, \emph{i.e.} $T(X)$ is the orthogonal complement of $\NS(X)$. We have a decomposition
\[
	\rH^2(X,\bQ(1)) = \NS(X)_\bQ \oplus T(X)_\bQ.
\]
The Hodge structure $T(X)_\bQ$ is irreducible \cite[Thm.~1.4.1]{Zarhin83}. 
The cup product pairing defines even symmetric bilinear forms on $\NS(X)$ and $T(X)$ of signature $(1,\rho-1)$ and $(2,20-\rho)$, with $\rho=\rk \NS(X)$.

\begin{proposition}Let $X$ be a projective K3 surface over $\bC$. Then the following are equivalent:
\begin{enumerate}
\item the Hodge group of $T(X)_\bQ$ is commutative,
\item $E:=\End_{\HS} T(X)_\bQ$ is a CM field and $\dim_E T(X)_\bQ = 1$.
\end{enumerate}
\end{proposition}

\begin{proof}\cite[\S 2]{Zarhin83}
\end{proof}

\begin{definition}\label{def:CM}
If $X$ satisfies the equivalent conditions (1) and (2) of the Proposition, then we say that $X$ is a K3 surface \emph{with CM} (by $E$).
\end{definition}

\begin{remark}Another equivalent condition is that $T(X)_\bQ$ is contained in the Tannakian category of Hodge structures generated by the $\rH^1$ of CM abelian varieties.
\end{remark}

If $E$ is a CM field, then we denote the canonical complex conjugation of $E$ by $z\mapsto \bar z$,
 and its fixed field by $E_0$. We have $[E:E_0]=2$, and $E_0$ is a totally real number field.
 
\begin{proposition}Let $X$ be a K3 surface with CM by $E$. Then 
\begin{enumerate}
\item $ax \cdot y = x \cdot \bar{a} y$ for all $a\in E$ and $x,y\in T(X)_\bQ$;
\item the group of Hodge isometries of $T(X)_\bQ$ is $\ker(\Nm \colon E^\times \to E_0^\times)$.
\end{enumerate}
\end{proposition} 

\begin{proof}The cup product pairing induces an isomorphism
\[
	T(X)_\bQ \isomto \Hom(T(X)_\bQ,\bQ)
\]
of Hodge structures, and hence the action of $E$ on $T(X)$ induces an `adjoint' homomorphism $\varphi\colon E\to E$ such that $ax\cdot y = x\cdot \varphi(a)y$. Considering the induced action on $\rH^{0,2}(X)$ one sees that $\varphi(a)=\bar{a}$, which proves the first assertion. The second is an immediate consequence of the first. 
\end{proof}

\subsection{Arithmetic aspects: the Main Theorem of CM}\label{subsec:main-thm-CM}
Let  $X$ be a K3 surface over $\bC$ with CM by $E$.
Consider the algebraic torus $G$ over $\bQ$ which is the kernel of the norm map $E^\times \to E_0^\times$ (seen as map of tori over $\bQ$). Then $G(\bQ)$ is the group of $E$-linear isometries of $T(X)_\bQ$. 

If $X$ is defined over a subfield $k\subset \bC$, then we have canonical isomorphisms
\[
	\rH^2_\et(X_{\bar k}, \bQ_\ell(1)) = \rH^2(X(\bC),\bQ(1)) \otimes_{\bQ} \bQ_\ell.
\]
Since the Galois action on the left-hand side respects the intersection pairing and the subgroup $\NS(X_{\bar k})=\NS(X_{\bC})$, we see that both $\Gal_k$ and $G(\bQ_\ell)$ act on $T(X)_{\bQ_\ell}$.
If we denote by $\bA_f$  the finite ad\`eles of $\bQ$, \emph{i.e.} $\bA_f = \bQ\otimes \hat\bZ$, then we obtain actions of $\Gal_k$ and $G(\bA_f)$ on $T(X)_{\bA_f}$.

\begin{theorem}[Rizov, Main theorem of CM for K3 surfaces]\label{thm:main-thm-CM} There exists a number field $k \subset \bC$ containing $E$ such that 
\begin{enumerate}
\item $X$ is defined over $k$,
\item the Galois action on $T(X)_{\bA_f}$ factors over a map $\rho\colon \Gal_k \to G(\bA_f)$
\item the diagram
\[
\begin{tikzcd}
\Gal_k \arrow{d}{\rho} \arrow[tail]{r}
	& \Gal_E \arrow[two heads]{r}{\text{CFT}}
	& \bA_{E,f}^\times/E^\times \arrow{d}{z \mapsto \bar z/z}\\
	G(\bA_f) \arrow{rr} 
	& 
	& G(\bA_f)/G(\bQ)
\end{tikzcd}
\]
commutes.
\end{enumerate}
\end{theorem}

\begin{proof}
This is a reformulation of \cite[Cor.~3.9.2]{Rizov10}. Note however that in \cite[1.4.3]{Rizov10} the definition of complex multiplication needs to be corrected (the condition $\dim_E T_\bQ=1$ is missing) for proof and statement to be correct.
\end{proof}

%
%
\begin{remark}
A priori, the moduli space of polarized complex K3 surfaces has two natural models over $\bQ$: the `canonical model' of the theory of Shimura varieties \cite[\S 3]{Deligne71}, which is defined in terms of the Galois action on special points, and the model coming from the moduli interpretation. The essential content of Theorem \ref{thm:main-thm-CM} is that these two models coincide.  (See also \cite[\S 3]{MadapusiPera15}).
\end{remark}

\section{Existence of K3 surface with CM by a given CM field}
\label{sec:CM-existence}

In this section we prove Theorem \ref{thm:K3-with-given-CM-field}. By the surjectivity of the period map for K3 surfaces, this reduces  to a problem about quadratic forms over $\bQ$.

\subsection{Invariants of quadratic forms over $\bQ$}
We quickly recall some basic facts about quadratic forms over $\bQ$. We refer to \cite{Cassels78,Scharlau85,Serre70} for details and proofs.
Let $k$ be a field of characteristic different from $2$. A \emph{quadratic space} over $k$ is a pair $V=(V,q)$ consisting of a finite-dimensional vector space over $k$ and a non-degenerate symmetric bilinear form $q\colon V\times V\to k$. To such a space one associates the following invariants:
\begin{enumerate}
\item the \emph{dimension} $\dim(V)$;
\item the \emph{determinant} $\det(V) \in k^\times/k^{\times2}$;
\item the \emph{Hasse invariant} $w(V) \in \Br(k)[2]$.
\end{enumerate}
Any form $V$ over $k$ is isomorphic to a diagonal form $\langle \alpha_1,\ldots,\alpha_n\rangle$ with $n=\dim V$, and for such a form the invariants are
\begin{eqnarray*}\label{eq:diagonal-det}
	\det(V) = \prod_i \alpha_i &\in&  k^\times/k^{\times2}, \\ \label{eq:diagonal-hasse}
	w(V) = \sum_{i<j} \,(\alpha_i,\alpha_j)_k &\in& \Br(k)[2],
\end{eqnarray*}
where $(\alpha,\beta)_k$ denotes the class of the quaternion algebra generated by $i$ and $j$ with $i^2=\alpha$, $j^2=\beta$, $ij=-ji$.

We denote the orthogonal sum of two quadratic spaces by $V\oplus W$. 

\begin{lemma}\label{lemma:additivity} Let $V$ and $W$ be quadratic spaces over $k$. Then
\begin{enumerate}
\item $\det(V\oplus W) = \det(V)\det(W)$;
\item $w(V\oplus W) = w(V) + w(W) + (\det(V),\det(W))_k$. 
\end{enumerate}
\end{lemma}

\begin{proof}
This follows from the above formulas for the determinant and Hasse invariant of a diagonal quadratic form, and the bilinearity of $(\alpha,\beta)_k$.
\end{proof}

\begin{theorem}\label{thm:local-classification}
Two forms over $\bQ_p$ are isomorphic if and only if they have the same dimension, determinant and Hasse invariant. For every $d\geq 3$, $\delta \in \bQ_p^\times/\bQ_p^{\times2}$ and $w\in \Br(\bQ_p)[2]$ there exists a form of dimension $d$,  determinant $\delta$ and Hasse invariant $w$. \qed
\end{theorem}

If $k=\bQ$ then a fourth invariant is given by the signature of the form $V_\bR$. 

\begin{theorem}\label{thm:global-classification}
Two forms over  $\bQ$  are isomorphic if and only if they have the same  signature, determinant, and Hasse invariant. All forms $V$ over $\bQ$ of signature $(r,s)$ satisfy
\begin{enumerate}
\item the sign of $\delta(V)$ is $(-1)^s$;
\item the image of $w(V)$ in $\Br(\bR)[2]=\bZ/2\bZ$ is $s(s-1)/2 \bmod 2$.
\end{enumerate}
If $r+s\geq 3$, and if $\delta$ and $w$ satisfy (1) and (2) above, then there exists a quadratic space over $\bQ$ with signature $(r,s)$, determinant $\delta$ and Hasse invariant $w$.\qed
\end{theorem}

Finally, we will need the invariants of $\Lambda_{K3,\bQ}=\bQ\otimes_\bZ \Lambda_{K3}$, which are as follows.

\begin{lemma}\label{lemma:invariants-K3} $\det(\Lambda_{K3,\bQ})=-1$ and $w(\Lambda_{K3,\bQ})\in \Br(\bQ)[2]$ is the class of the quaternion algebra $(-1,-1)_\bQ$.
\end{lemma}

\begin{proof}We have $\Lambda_{K3} \cong (-E_8) \oplus (-E_8) \oplus U \oplus U \oplus U$ where $U$ is the standard hyperbolic plane. Using this explicit description, one computes (over $\bQ$) an orthogonal basis, and  computes the invariants using the formula for diagonal forms.
\end{proof}

\subsection{The form $q_\lambda$}

Let $E$ be a CM field with maximal totally real subfield $E_0$. Put $d:= [E_0:\bQ]$.
Denote by $z\mapsto \bar{z}$  the complex conjugation on $E$. For $\lambda \in E_0^\times$ the map
\[
	q_\lambda\colon E\times E \mapsto \bQ,\, (x,y) \mapsto \tr_{E_0/\bQ}(\lambda x\bar{y})
\]
is a non-degenerate symmetric bilinear form over $\bQ$. 

We denote the discriminant of the number field $E$ by $\Delta(E/\bQ)$.

\begin{lemma}\label{lemma:disc-of-q}
$\det(q_\lambda) = (-1)^d \,\Delta(E/\bQ)$ in $\bQ^\times/\bQ^{\times 2}$.
\end{lemma}

\begin{proof}See~\cite[Lemma 1.3.2]{Bayer14}.
\end{proof}

\begin{lemma}\label{lemma:sign-of-q}
If $\lambda \in E_0^\times$ has signature $(r,s)$, then $q_\lambda$ has signature $(2r,2s)$.
\qed
\end{lemma}

\subsection{Construction of a K3 surface with CM by $E$}
A key ingredient in the proof of Theorem \ref{thm:K3-with-given-CM-field} is the following proposition on rational quadratic forms. I am grateful to Eva Bayer for pointing me to her work on maximal tori in orthogonal groups \cite{Bayer14}, and for explaining how it simplifies an earlier version of the proof below.

\begin{proposition}\label{prop:CM-to-K3}
Let $E$ be a CM field with maximal totally real subfield $E_0$, and assume $d:=[E_0: \bQ] \leq 10$.  Then there exists a $\lambda \in E_0^\times$ of signature $(1,d-1)$ and a quadratic space $V$ such that 
\begin{equation*}
	(E, q_\lambda) \,\oplus\, V \,\cong\, \Lambda_{K3,\bQ}
\end{equation*}
as quadratic spaces over $\bQ$.
\end{proposition}

\begin{proof}
If $d<10$ then we claim that for every choice of $\lambda$ a complement $V$ exists. Indeed, given a choice of $\lambda$, then  the dimension, signature, determinant and Hasse invariant of $V$ are determined by Lemma \ref{lemma:additivity}. These invariants satisfy conditions (1) and (2) of  Theorem \ref{thm:global-classification} because they are satisfied by the invariants of $(E,q_\lambda)$ and $\Lambda_{K3,\bQ}$.  Since $\dim(V)>2$, the theorem then guarantees the existence of a form $V$ with $(E, q_\lambda) \oplus V \cong \Lambda_{K3,\bQ}$.

So we assume $d=10$. Let $\delta = \Delta(E/\bQ) \in \bQ^\times/\bQ^{\times 2}$. Note that $\delta>0$ (since $d$ is even).
Consider the diagonal quadratic space $V=\langle -1,\delta\rangle$. By the same reasoning as above, there exists a unique quadratic space $W$ of dimension $20$ such that
\[
	W \,\oplus\, \langle -1,\delta\rangle \,\cong\, \Lambda_{K3,\bQ}.
\]
We will show that $W$ can be realized as $(E,q_\lambda)$ for a suitable choice of
$\lambda\in E_0^\times$. Note that $W$ has signature $(2,18)$, so by Lemma \ref{lemma:sign-of-q} the scalar $\lambda$ will automatically have signature $(1,9)$.

By Cor.~4.0.3 and Prop.~1.3.1 of \cite{Bayer14}, there exists a $\lambda$ with $(E,q_\lambda)\cong W$ if and only if the following three conditions hold
\begin{enumerate}
\item the signature of $W$ is even
\item $\disc(W) = \delta$
\item for every prime $p$ such that all places of $E_0$ above $p$ split in $E$ we have that $W_{\bQ_p}$ is isomorphic to an orthogonal sum of $10$ hyperbolic planes.
\end{enumerate}
Our $W$ clearly satisfies the first two conditions. For the third, consider a prime $p$  such that all places of $E_0$ above $p$ split in $E$. Then the image  of $\delta$ in $\bQ_p^\times/\bQ_p^{\times2}$ is $1$. Together with Lemma \ref{lemma:additivity} and Lemma \ref{lemma:invariants-K3} this allows us to compute the invariants of $W_{\bQ_p}$, and we find $\det(W_{\bQ_p}) = 1$ and $w(W_{\bQ_p}) = (-1,-1)_{\bQ_p}$.
These are the same as the invariants for $10$ copies of the hyperbolic plane, so with Theorem \ref{thm:local-classification} we see that $W$ satisfies the third condition, which finishes the proof of the proposition.
\end{proof}

Finally, we show that for every CM field $E$ of degree at most 20 there exists a projective K3 surface $X$ with CM by $E$.

\begin{proof}[Proof of Theorem \ref{thm:K3-with-given-CM-field}]
Choose $\lambda \in E_0$ and $V$ as in Proposition \ref{prop:CM-to-K3}. This guarantees that there exists an integral lattice
\[
	\Lambda \subset (E,q_\lambda) \oplus V
\]
with $\Lambda \cong \Lambda_{K3}$. Choose such a $\Lambda$, and choose an embedding $\epsilon\colon E \injto \bC$ with $\epsilon(\lambda) >0$. Then we have a splitting
\[
	\Lambda_\bC \,=\, \bC_{\epsilon} \,\oplus\, \bC_{\bar\epsilon} 
		\,\oplus\, (\oplus_{\sigma\not= \epsilon,\bar\epsilon} \bC_\sigma)
		\,\oplus\, V_{\bC}.
\]
We make $\Lambda$ into a pure $\bZ$-Hodge structure of weight $0$ by declaring $\bC_{\epsilon}$ to be of type $(1,-1)$, its conjugate $\bC_{\bar\epsilon}$ of type $(-1,1)$, and all the other terms of type $(0,0)$. By construction, the bilinear form $\Lambda \otimes \Lambda \to \bZ$ is a morphism of Hodge structures. Note that $E$ acts on $E \subset \Lambda_\bQ$ via Hodge structure endomorphisms, so that $E$ is irreducible and hence
\[
	\Lambda^{0,0} \cap \Lambda_\bQ = V.
\]
For every non-zero $z \in \rH^{2,0}$ we have $z\cdot \bar{z} \in \bR_{>0}$ since $\epsilon(\lambda)>0$, 
so that the surjectivity of the period map \cite{Todorov80} gives the existence of a complex analytic K3 surface $X$ and a Hodge isometry $\Lambda \cong \rH^2(X,\bZ(1))$. A priori, it may not be clear that $X$ is algebraic. However, as 
 $\Pic(X)_{\bQ} \cong V$ has signature $(1,21-2d)$, there exists an $h\in \Pic(X)$ with $h\cdot h > 0$. By \cite[Thm.~IV.6.2]{BHPV} this implies that the surface $X$ is projective. By construction, $X$ is a K3 surface with CM by $E$.
\end{proof}

\begin{remark}A similar construction has been used by
Piatetski-Shapiro and Shafarevich \cite[\S 3]{PiatetskiShapiroShafarevich73} in showing the existence of some K3 surfaces with CM. The new ingredients that allow us to obtain a stronger result are the use of rational (as opposed to integral) quadratic forms, the results of Bayer on quadratic forms $q_\lambda$, and the use of the algebraicity criterion from \cite{BHPV}, which avoids the delicate question of identifying an ample $h\in \Pic(X)$.
\end{remark}

\section{Existence of K3 surface with given $L_\trc$}

In this section we will prove Theorem \ref{thm:converse}. So
let
\[
	L = \prod_{i=1}^{2d} (1-\gamma_iT) \in 1+T\bQ[T]
\]
be a polynomial satisfying properties (1)--(5) of Theorem \ref{thm:properties}. 
Consider the number field $F:=\bQ(\gamma_1)$. 

\begin{lemma}$F$ is a CM field and $\bar\gamma_1\gamma_1=1$.
\end{lemma}

\begin{proof}
The image $\gamma$ of $\gamma_1$ under any homomorphism $F\to\bC$ satisfies
$|\gamma|=1$, hence $\bar\gamma=\gamma^{-1}$. Moreover  $\gamma$ cannot be real, since then $\gamma=\pm 1$, contradicting the fact that $\gamma_1$ is not a root of unity.  It follows that $F$ is a CM field with complex conjugation $\gamma_1 \mapsto \gamma_1^{-1}$.
\end{proof}

By property (5), the number field $F$ has a unique valuation $v$ above the prime $p$ such that $v(\gamma_1)<0$. 

\begin{lemma}There exists an extension $E$ of $F$ with $[E:\bQ]=2d$, and such that
\begin{enumerate}
\item $E$ is a CM field;
\item the valuation $v$ has a unique extension to $E$. 
\end{enumerate}
\end{lemma}

\begin{proof}
Let $F_0$ be the maximal totally real subfield of $F$. Let $v_0$ be the place of $F_0$ under $v$.  Now choose a  polynomial $P(X)\in F_0[X]$  
such that
\begin{enumerate}
\item $\deg P = e$;
\item $P$ has $e$ real roots for every embedding $F_0 \injto \bR$;
\item $P$ is irreducible in $(F_0)_{v_0}[X]$.
\end{enumerate}
Note that $v_0$ splits in $F$, since by the preceding lemma $\bar{v}(\gamma_1)>0$ and hence $\bar{v} \neq v$. In particular $P(X)$ is irreducible in $F_v[X]$, and it follows that $E := F[X]/P(X)$ is a field satisfying the desired conditions. 
\end{proof}

We fix an $E$ satisfying the conditions of the lemma. Abusing notation, we will denote the unique extension of $v$ to $E$ by the same symbol $v$.

\begin{lemma}$[E_v:\bQ_p]=h$.
\end{lemma}

\begin{proof}
Since $L=Q^e$, and since $v$ is the unique place with $v(\gamma_1)<0$, we see from properties (4) and (5) in Theorem \ref{thm:properties} that $[F_v:\bQ_p]=h/e$. But $[E:F]=e$ and $v$ has a unique extension to $E$, hence $[E_v:\bQ_p]=h$.
\end{proof}

Let $X$ be a K3 surface over $\bC$ with CM by $E$. By the Main Theorem of CM (Theorem \ref{thm:main-thm-CM}) this surface is defined over a number field $k$ containing $E$. Let $w$ be a place of $k$ lying above $v$. We extend the commutative diagram of Theorem \ref{thm:main-thm-CM}  to include the local-global compatibility of class field theory: \begin{equation*}\label{eq:big-diagram}
\begin{tikzcd}
W_{k_w} \arrow[tail]{d} \arrow[tail]{r} 
	& W_{E_v} \arrow[tail]{d} \arrow[two heads]{r}{\text{LCFT}}
	& E_v^\times \arrow{d} \\
\Gal_k \arrow{d}{\rho} \arrow[tail]{r}
	& \Gal_E \arrow[two heads]{r}{\text{GCFT}}
	& \bA_{E,f}^\times/E^\times \arrow{d}{z \mapsto  z/\bar z}\\
	G(\bA_f) \arrow{rr} 
	& 
	& G(\bA_f)/G(\bQ)
\end{tikzcd}
\end{equation*}
Here $W_{k_w}\subset \Gal_{k_w}$ denotes the Weil group of the local field $k_w$. Extending $k$ if necessary, we may assume that the residue field $\bF_w$ is an extension of $\bF_q$.

Choose a prime $\ell\neq p$. Then the image of inertia $I_{k_w}$ in $G(\bZ_\ell)$ is finite,
hence after replacing $k$ by a finite extension, we may assume that the action of $\Gal_{k_w}$ on $\rH^2_\et(X_{\bar{k}},\bQ_\ell(1))$ is unramified. 

Now assume $X_{k_w}$ satisfies ($\star$). Then, replacing $k$ once more by a finite extension, we may assume by the criterion of Liedtke and Matsumoto \cite[Thm~2.5]{LiedtkeMatsumoto15} that $X$ has good reduction at $w$. Let $\bar{X}/\bF_w$ be the reduction of $X/k$ at $w$.

Let $\sigma \in W_{k_w}$ be a Frobenius element. Note that $\gamma_1$ lies in $G(\bQ)\subset E^\times$.

\begin{proposition}
There is an $m>0$ such that  for all $\ell\neq p$ we have
\[
	\rho(\sigma^m)_\ell = \gamma_1^{m [\bF_w:\bF_q]}
\]
in $G(\bQ_\ell)$.
\end{proposition}

\begin{proof}
Let $\pi \in E_v^\times$ be the image of $\sigma$ under the CFT map. Then
\[
	v(\pi)=\frac{e(k_w:E_v)}{f(E_v:\bQ_p)}.
\]
The image of $\pi$ in $G(\bA_f)/G(\bQ)$ is the class of the id\`ele 
\[
	(1,\ldots, 1, \pi, \bar\pi^{-1}, 1, \ldots, 1 ) \in \bA_{E,f}^\times,
\]
where $\bar\pi \in E_{\bar v}$ denotes the image of $\pi$ under the isomorphism $E_v \to E_{\bar v}$ induced by complex conjugation on $E$.

We have $v(\gamma_1)=-[\bF_q:\bF_p]/h$ from which we compute
\[
	v(\gamma_1^{[\bF_w:\bF_q]})=-v(\pi)
\]
and hence $\bar{v}(\gamma_1^{[\bF_w:\bF_q]})=v(\pi)$. Moreover, $\gamma_1$ is a unit at all places of $E$ different from $v$ and $\bar{v}$. It follows that the id\`ele
\[
	\alpha := \gamma_1^{[\bF_w:\bF_q]} \cdot
	(1,\ldots,1,\pi,\bar\pi,1,\ldots, 1) \in \bA_{E,f}^\times.
\]
lies in the maximal compact subgroup 
\[
	\cK = \{ g \in (\cO_E\otimes \hat\bZ)^\times \mid g\bar{g} = 1 \} \subset G(\bA_f).
\]
Since $\Gal_k$ is compact also $\rho(\sigma)$ lies in $\cK$. From the commutativity of the diagram (\ref{eq:big-diagram}) we conclude that
$\rho(\sigma)/\alpha$ lies in the kernel of the map
\[
	\cK \to G(\bA_f)/G(\bQ).
\]
This kernel equals $\{ g \in \cO_E^\times \mid g\bar{g}=1 \}$,
which is finite by the Dirichlet unit theorem.
We conclude that
$\rho(\sigma^m) = \alpha^m$ for some $m$, and hence
\[
	\rho(\sigma^m)_\ell = \gamma_1^m
\]
in $G(\bQ_\ell)$ for all $\ell \neq p$.
\end{proof}

We have
\[
	L(\bar{X}/\bF_w) = {\det}_{\bQ_\ell}(1-\sigma T,\, \rH^2_\et(X_{\bar \bF_w}, \bQ_\ell(1)) ).
\]
Since none of the conjugates of $\gamma_1$ are roots of unity, we conclude with the preceding proposition that there is a finite extension $\bF_w\subset \bF_w'$  such that
\[
	L_\trc(\bar{X}_{\bF_w'}/\bF_w') = {\det}_\bQ(1-\gamma_1^{[\bF_w':\bF_q]} T,\, E),
\]
or in other words:
\[
	L_\trc(\bar{X}_{\bF_w'}/\bF_w') = \prod_i (1-\gamma_i^{[\bF_w':\bF_q]} T),
\]
which finishes the proof of Theorem \ref{thm:converse}.

\end{document}